\documentclass{amsart}
\usepackage{amssymb}
\usepackage{amsmath}
\usepackage{amsfonts}

\newtheorem{theorem}{Theorem}
\theoremstyle{plain}

\newtheorem{corollary}{Corollary}

\newtheorem{example}{Example}

\newtheorem{lemma}{Lemma}

\numberwithin{equation}{section}
\input{tcilatex}

\begin{document}
\title[Characterizations of AG-groupoids by their intuitionistic fuzzy ideals%
]{Characterizations of Abel-Grassmann's groupoids by their intuitionistic
fuzzy ideals}
\author{}
\maketitle

\begin{center}
\textbf{Madad Khan and Faisal}

\textbf{Department of Mathematics}

\textbf{COMSATS Institute of Information Technology}

\textbf{Abbottabad, Pakistan.}

\textbf{E-mail: madadmath@yahoo.com}

\textbf{E-mail: yousafzaimath@yahoo.com}
\end{center}

\QTP{Body Math}
$\bigskip $

\textbf{Abstract. }In this paper, we have introduced the concept of
intuitionistic fuzzy ideals in an AG-groupoids. We have characterized
regular and intra-regular AG-groupoids in terms of intuitionistic fuzzy left
(right, two-sided) ideals, fuzzy (generalized) bi-ideals and intuitionistic
fuzzy $(1,2)$-ideals. We have proved that all the intuitionistic fuzzy
ideals coincides in regular and intra-regular AG-groupoids. It has been
shown that the set of intuitionistic fuzzy two sided ideals of a regular
AG-groupoid forms a semilattice structure. We have also given some useful
conditions for an AG-groupoid to become an intra-regular AG-groupoid in
terms their intuitionistic fuzzy ideals.

AG-groupoids and intuitionistic fuzzy ideals.

\begin{center}
\bigskip

{\LARGE Introduction}
\end{center}

The concept of fuzzy sets was first proposed by Zadeh \cite{L.A.Zadeh} in $%
1965$. Several researchers were conducted on the generalization of the
notion of fuzzy set. Given a set $S$, a fuzzy subset of $S$ is, by
definition an arbitrary mapping $f:S\rightarrow \lbrack 0,1]$ where $[0,1]$
is the unit interval. A fuzzy subset is a class of objects with a grades of
membership. The concept of fuzzy set was applied in \cite{cha} to generalize
the basic concepts of general topology. A. Rosenfeld \cite{A. Rosenfeld} was
the first who consider the case of a groupoid in terms of fuzzy sets. Kuroki
has been first studied the fuzzy sets on semigroups \cite{N. Kuroki}.

As an important generalization of the notion of fuzzy set , Atanassov \cite%
{at}, introduced the concept of an intuitionistic fuzzy set. De et al. \cite%
{7} studied the Sanchez's approach for medical diagnosis and extended this
concept with the notion of intuitionistic fuzzy set theory. Dengfeng and
Chunfian \cite{8} introduced the concept of the degree of similarity between
intuitionistic fuzzy sets, which may be finite or continuous, and gave
corresponding proofs of these similarity measure and discussed applications
of the similarity measures between intuitionistic fuzzy sets to pattern
recognition problems. Intuitionistic fuzzy sets have many applications in
mathematics, Davvaz et al. \cite{6}, applied this concept in $H_{v}$%
-modules. They introduced the notion of an intuitionistic fuzzy $H_{v}$%
-submodule of an $H_{v}$-module and studied the related properties. Jun in 
\cite{10}, introduced the concept of an intuitionistic fuzzy bi-ideal in
ordered semigroups and characterized the basic properties of ordered
semigroups in terms of intuitionistic fuzzy bi-ideals. In \cite{15} and \cite%
{16}, Kim and Jun introduced the concept of intuitionistic fuzzy interior
ideals of semigroups. In \cite{20}, Shabir and Khan gave the concept of an
intuitionistic fuzzy interior ideal of ordered semigroups and characterized
different classes of ordered semigroups in terms of intuitionistic fuzzy
interior ideals. They also gave the concept of an intuitionistic fuzzy
generalized bi-ideal in \cite{21} and discussed different classes of ordered
semigroups in terms of intuitionistic fuzzy generalized bi-ideals.

In this paper, we consider the intuitionistic fuzzification of the concept
of several ideals in AG-groupoid and investigate some properties of such
ideals.

An AG-groupoid is a non-associative algebraic structure mid way between a
groupoid and a commutative semigroup. The left identity in an AG-groupoid if
exists is unique \cite{Mus3}. An AG-groupoid is non-associative and
non-commutative algebraic structure, nevertheless, it posses many
interesting properties which we usually find in associative and commutative
algebraic structures. An AG-groupoid with right identity becomes a
commutative monoid \cite{Mus3}. An AG-groupoid is basically the
generalization of semigroup (see \cite{Kaz}) with wide range of applications
in theory of flocks \cite{nas}. The theory of flocks tries to describes the
human behavior and interaction.

The concept of an Abel-Grassmann's groupoid (AG-groupoid) \cite{Kaz} was
first given by M. A. Kazim and M. Naseeruddin in $1972$ and they called it
left almost semigroup (LA-semigroup). P. Holgate call it simple invertive
groupoid \cite{hol}. An AG-groupoid is a groupoid having the left invertive
law,%
\begin{equation}
(ab)c=(cb)a\text{, for\ all }a\text{, }b\text{, }c\in S\text{.}  \tag{$1$}
\end{equation}%
In an AG-groupoid, the medial law \cite{Kaz} holds,%
\begin{equation}
(ab)(cd)=(ac)(bd)\text{, for\ all }a\text{, }b\text{, }c\text{, }d\in S\text{%
.}  \tag{$2$}
\end{equation}%
In an AG-groupoid $S$ with left identity, the paramedial law holds,%
\begin{equation}
(ab)(cd)=(dc)(ba),\text{ for\ all }a,b,c,d\in S.  \tag{$3$}
\end{equation}%
If an AG-groupoid contains a left identity, the following law holds,

\begin{equation}
a(bc)=b(ac)\text{, for\ all }a\text{, }b\text{, }c\in S\text{.}  \tag{$4$}
\end{equation}

\begin{center}
\bigskip

{\LARGE Preliminaries}
\end{center}

Let $S$ be an AG-groupoid, by an AG-subgroupoid of $S,$ we means a non-empty
subset $A$ of $S$ such that $A^{2}\subseteq A$.

A non-empty subset $A$ of an AG-groupoid $S$ is called left (right) ideal of 
$S$ if $SA\subseteq A$ $(AS\subseteq A)$.

A non-empty subset $A$ of an AG-groupoid $S$ is called two-sided ideal or
simply ideal if it is both a left and a right ideal of $S$.

A non empty subset $A$ of an AG-groupoid $S$ is called generalized bi-ideal
of $S$ if $(AS)A\subseteq A$.

An AG-subgroupoid $A$ of $S$ is called bi-ideal of $S$ if $(AS)A\subseteq A$.

An AG-subgroupoid $A$ of an AG-groupoid $S$ is called $(1,2)$ ideal of $S$
if $(AS)A^{2}\subseteq A.$

\bigskip

Let $S$ be a non empty set, a fuzzy subset $f$ of $S$ is, by definition an
arbitrary mapping $f:S\rightarrow \lbrack 0,1]$ where $[0,1]$ is the unit
interval. A fuzzy subset $f$ is a class of objects with a grades of
membership having the form

\begin{center}
$f=\{(x,$ $f(x))/x\in S\}.$
\end{center}

An intuitionistic fuzzy set (briefly, $IFS$) $A$ in a non empty set $S$ is
an object having the form

\begin{center}
$A=\left\{ (x,\mu _{A}(x),\gamma _{A}(x))/x\in S\right\} .$
\end{center}

The functions $\mu _{A}:S\longrightarrow \lbrack 0,1]$ and $\gamma
_{A}:S\longrightarrow \lbrack 0,1]$ denote the degree of membership and the
degree of nonmembership respectively such that for all $x\in S,$ we have

\begin{center}
$0\leq \mu _{A}(x)+\gamma _{A}(x)\leq 1.$
\end{center}

\QTP{Body Math}
For the sake of simplicity, we shall use the symbol $A=(\mu _{A},\gamma
_{A}) $ for the $IFS$ $A=\left\{ (x,\mu _{A}(x),\gamma _{A}(x))/x\in
S\right\} .$

\QTP{Body Math}
Let $\delta =\left\{ (x,S_{\delta }(x),\Theta _{\delta }(x))/S_{\delta }(x)=1%
\text{ and }\Theta _{\delta }(x)=0\text{ for all }x\in S\right\} =(S_{\delta
},\Theta _{\delta })$ be an $IFS,$ then $\delta =(S_{\delta },\Theta
_{\delta })$ will be carried out in operations with an $IFS$ $A=(\mu
_{A},\gamma _{A})$ such that $S_{\delta }$ and $\Theta _{\delta }$ will be
used in collaboration with $\mu _{A}$ and $\gamma _{A}$ respectively.

An $IFS$ $A=(\mu _{A},\gamma _{A})$ of an AG-groupoid $S$ is called an
intuitionistic fuzzy AG-subgroupoid of $S$ if $\mu _{A}(xy)\geq \mu
_{A}(x)\wedge \mu _{A}(y)$ and $\gamma _{A}(xy)\leq \gamma _{A}(x)\vee
\gamma _{A}(y)$ for all $x$, $y\in S.$

An $IFS$ $A=(\mu _{A},\gamma _{A})$ of an AG-groupoid $S$ is called an
intuitionistic fuzzy left ideal of $S$ if $\mu _{A}(xy)\geq \mu _{A}(y)$ and 
$\gamma _{A}(xy)\leq \gamma _{A}(y)$ for all $x$, $y\in S.$

An $IFS$ $A=(\mu _{A},\gamma _{A})$ of an AG-groupoid $S$ is called an
intuitionistic fuzzy right ideal of $S$ if $\mu _{A}(xy)\geq \mu _{A}(x)$
and $\gamma _{A}(xy)\leq \gamma _{A}(x)$ for all $x$, $y\in S.$

An $IFS$ $A=(\mu _{A},\gamma _{A})$ of an AG-groupoid $S$ is called fuzzy
two-sided ideal of $S$ if it is both an intuitionistic fuzzy left and an
intuitionistic fuzzy right ideal of $S$.

An $IFS$ $A=(\mu _{A},\gamma _{A})$ of an AG-groupoid $S$ is called an
intuitionistic fuzzy generalized bi-ideal of $S$ if $\mu _{A}((xa)y)\geq \mu
_{A}(x)\wedge \mu _{A}(y)$ and $\gamma _{A}((xa)y)\leq \gamma _{A}(x)\vee
\gamma _{A}(y)$ for all $x$, $a$ and $y\in S$.

An intuitionistic fuzzy AG-subgroupoid $A=(\mu _{A},\gamma _{A})$ of an
AG-groupoid $S$ is called an intuitionistic fuzzy bi-ideal of $S$ if $\mu
_{A}((xa)y)\geq \mu _{A}(x)\wedge \mu _{A}(y)$ and $\gamma _{A}((xa)y)\leq
\gamma _{A}(x)\vee \gamma _{A}(y)$ for all $x$, $a$ and $y\in S$.

An intuitionistic fuzzy AG-subgroupoid $A=(\mu _{A},\gamma _{A})$ of an
AG-groupoid $S$ is called an intuitionistic fuzzy $(1,2)$-ideal of $S$ if $%
\mu _{A}((xw)(yz))\geq \mu _{A}(x)\wedge \mu _{A}(y)\wedge \mu _{A}(z)$ and $%
\gamma _{A}((xw)(yz))\leq \gamma _{A}(x)\vee \gamma _{A}(y)\vee \gamma
_{A}(z)$ for all $x$, $a$ and $y\in S$.

Let $S$ be an AG-groupoid and let $\phi \neq A\subseteq S,$ then the
intuitionistic characteristic function $\chi _{A}=(\mu _{\chi _{A}},\gamma
_{\chi _{A}})$ of $A$ is defined as

\begin{center}
$\mu _{\chi _{A}}(x)=\left\{ 
\begin{array}{c}
1\text{, if }x\in A \\ 
0\text{, if }x\notin A%
\end{array}%
\right. $ and $\gamma _{\chi _{A}}(x)=\left\{ 
\begin{array}{c}
0\text{, if }x\in A \\ 
1\text{, if }x\notin A%
\end{array}%
\right. $
\end{center}

\QTP{Body Math}
It is clear that $\gamma _{\chi _{A}}$ acts as a complement of $\mu _{\chi
_{A}},$ that is, $\gamma _{\chi _{A}}=\mu _{\chi _{A^{C}}}.$

\QTP{Body Math}
Note that in an AG-groupoid $S$ with left identity, $S=S^{2}.$

An element $a$ of an AG-groupoid $S$ is called regular if there exists $x\in
S$ such that $a=(ax)a$ and $S$ is called regular if every element of $S$ is
regular.

\begin{example}
\label{ex}Let $S=\{a,b,c,d,e\}$ be an AG-groupoid with left identity $d$
with the following multiplication table.
\end{example}

\begin{center}
\begin{tabular}{l|lllll}
. & $a$ & $b$ & $c$ & $d$ & $e$ \\ \hline
$a$ & $a$ & $a$ & $a$ & $a$ & $a$ \\ 
$b$ & $a$ & $b$ & $b$ & $b$ & $b$ \\ 
$c$ & $a$ & $b$ & $d$ & $e$ & $c$ \\ 
$d$ & $a$ & $b$ & $c$ & $d$ & $e$ \\ 
$e$ & $a$ & $b$ & $e$ & $c$ & $d$%
\end{tabular}

By routine calculation, it is easy to check that $S$ is regular.
\end{center}

Define an $IFS$ $A=(\mu _{A},\gamma _{A})$ of $S$ as follows: $\mu _{A}(a)=1$%
, $\mu _{A}(b)=$ $\mu _{A}(c)=$ $\mu _{A}(d)=$ $\mu _{A}(e)=0,$ $\gamma
_{A}(a)=0.3,$ $\gamma _{A}(b)=0.4$ and $\gamma _{A}(c)=\gamma _{A}(d)=\gamma
_{A}(e)=0.2,$ then clearly $A=(\mu _{A},\gamma _{A})$ is an intuitionistic
fuzzy two-sided ideal of $S$.

\begin{theorem}
\label{aw}Let $S$ be an AG-groupoid with left identity and let $A=(\mu
_{A},\gamma _{A})$ be any $IFS$ of $S$, then $S$ is regular if $%
A(x)=A(x^{2}) $ holds for all $x$ in $S$.
\end{theorem}

\begin{proof}
Assume that $S$ be an AG-groupoid with left identity. Clearly $x^{2}S$ is a
subset of $S$ and therefore its characteristic function $\chi _{x^{2}S}=(\mu
_{\chi _{x^{2}S}},\gamma _{\chi _{x^{2}S}})$ is an $IFS$ of $S$. Let $x\in
S, $ then by given assumption $\mu _{\chi _{x^{2}S}}(x)=\mu _{\chi
_{x^{2}S}}(x^{2})$ and $\gamma _{\chi _{x^{2}S}}(x)=\gamma _{\chi
_{x^{2}S}}(x^{2})$ holds for all $x\in S.$ As $x^{2}\in x^{2}S,$ because by
using $(3)$, we have 
\begin{equation*}
x^{2}S=(xx)(SS)=(SS)(xx)=Sx^{2}.
\end{equation*}%
Therefore $\mu _{\chi _{x^{2}S}}(x^{2})=1$ and $\gamma _{\chi
_{x^{2}S}}(x^{2})=0,$ which implies that $x\in x^{2}S.$ Now by using $(1),$ $%
(4)$ and $(3)$, we have%
\begin{eqnarray*}
x &\in &x^{2}S=(xx)(SS)=((SS)x)x\subseteq ((SS)(x^{2}S))x=((SS)((xx)S))x \\
&=&((SS)((Sx)x))x=((Sx)(Sx))x=((xS)(xS))x=(x((xS)S))x\subseteq (xS)x.
\end{eqnarray*}%
Thus $S$ is regular.
\end{proof}

The converse is not true in general. For this let us consider a regular
AG-groupoid $S$ in Example \ref{ex}. Define an $IFS$ $A=(\mu _{A},\gamma
_{A})$ of $S$ as follows: $\mu _{A}(a)=0.6,$ $\mu _{A}(b)=0.2$, $\mu
_{A}(c)=\mu _{A}(d)=\mu _{A}(e)=0.9,$ $\gamma _{A}(a)=0.7,$ $\gamma
_{A}(b)=0.3$ and $\gamma _{A}(c)=\gamma _{A}(d)=\gamma _{A}(e)=1,$ then it
is easy to see that $\mu _{A}(a)\neq \mu _{A}(a^{2})$ and $\gamma
_{A}(a)\neq \gamma _{A}(a^{2}),$ that is, $A(a)\neq A(a^{2})$ for $a\in S.$

Let $A=(\mu _{A},\gamma _{A})$ and $B=(\mu _{B},\gamma _{B})$ are $IFSs$ of
an AG-groupoid $S.$ The symbols $A\cap B$ will means the following $IFS$ of $%
S$

\begin{center}
$(\mu _{A}\cap \mu _{A})(x)=\min \{\mu _{A}(x),\mu _{A}(x)\}=\mu
_{A}(x)\wedge \mu _{A}(x),$ for all $x$ in $S.$

$(\gamma _{A}\cup \gamma _{A})(x)=\max \{\gamma _{A}(x),\gamma
_{A}(x)\}=\gamma _{A}(x)\vee \gamma _{A}(x),$ for all $x$ in $S.$
\end{center}

\QTP{Body Math}
The symbols $A\cup B$ will means the following $IFS$ of $S$

\begin{center}
$(\mu _{A}\cup \mu _{A})(x)=\max \{\mu _{A}(x),\mu _{A}(x)\}=\mu _{A}(x)\vee
\mu _{A}(x),$ for all $x$ in $S.$

$(\gamma _{A}\cap \gamma _{A})(x)=\min \{\gamma _{A}(x),\gamma
_{A}(x)\}=\gamma _{A}(x)\wedge \gamma _{A}(x),$ for all $x$ in $S.$
\end{center}

$A\subseteq B$ means that

\begin{center}
$\mu _{A}(x)\leq \mu _{A}(x)$ and $\gamma _{A}(x)\geq \gamma _{A}(x)$ for
all $x$ in $S.$
\end{center}

Let $A=(\mu _{A},\gamma _{A})$ and $B=(\mu _{B},\gamma _{B})$ be any two $%
IFSs$ of an AG-groupoid $S$, then the product $A\circ B$ is defined by,

\begin{equation*}
\left( \mu _{A}\circ \mu _{B}\right) (a)=\left\{ 
\begin{array}{c}
\dbigvee\limits_{a=bc}\left\{ \mu _{A}(b)\wedge \mu _{B}(c)\right\} \text{,
if }a=bc\text{ for some }b,\text{ }c\in S. \\ 
0,\text{ otherwise.}%
\end{array}%
\right.
\end{equation*}

\begin{equation*}
\left( \gamma _{A}\circ \gamma _{B}\right) (a)=\left\{ 
\begin{array}{c}
\dbigwedge\limits_{a=bc}\left\{ \gamma _{A}(b)\vee \gamma _{B}(c)\right\} 
\text{, if }a=bc\text{ for some }b,\text{ }c\in S. \\ 
1,\text{ otherwise.}%
\end{array}%
\right.
\end{equation*}

\begin{lemma}
$($\cite{Mordeson},\cite{mad}$)$ \label{as}Let $S$ be an AG-groupoid$,$ then
the following holds.
\end{lemma}

$(i)$ An $IFS$ $A=(\mu _{A},\gamma _{A})$ is an intuitionistic fuzzy
AG-subgroupoid of $S$ if and only if $\mu _{A}\circ \mu _{A}\subseteq \mu
_{A}$ and $\gamma _{A}\circ \gamma _{A}\supseteq \gamma _{A}.$

$(ii)$ An $IFS$ $A=(\mu _{A},\gamma _{A})$ is intuitionistic fuzzy left
(right) ideal of $S$ if and only if $S\circ \mu _{A}\subseteq \mu _{A}$ and $%
\Theta \circ \gamma _{A}\supseteq \gamma _{A}$ $(\mu _{A}\circ S\subseteq
\mu _{A}$ and $\gamma _{A}\circ \Theta \supseteq \gamma _{A}).$

\begin{lemma}
\label{fgh}Let $S$ be an AG-groupoid and let $A=(\mu _{A},\gamma _{A})$ and $%
B=(\mu _{B},\gamma _{B})$ are any intuitionistic fuzzy two sided ideals of $%
S,$ then $A\circ B=A\cap B$.
\end{lemma}

\begin{proof}
Assume that $A=(\mu _{A},\gamma _{A})$ and $B=(\mu _{B},\gamma _{B})$ are
any intuitionistic fuzzy two sided ideals of a regular AG-groupoid $S$, then
by using Lemma \ref{as}, we have $\mu _{A}\circ \mu _{B}\subseteq \mu
_{A}\cap \mu _{B}$ and $\gamma _{A}\circ \gamma _{B}\supseteq \gamma
_{A}\cup \gamma _{B},$ which shows that $A\circ B\subseteq A\cap B$. Let $%
a\in S,$ then there exists $x\in S$ such that $a=(ax)a$ and therefore, we
have%
\begin{eqnarray*}
(\mu _{A}\circ \mu _{B})(a) &=&\dbigvee\limits_{a=(ax)a}\{\mu _{A}(ax)\wedge
\mu _{B}(a)\}\geq \mu _{A}(ax)\wedge \mu _{B}(a) \\
&\geq &\mu _{A}(a)\wedge \mu _{B}(a)=(\mu _{A}\cap \mu _{B})(a)
\end{eqnarray*}

and%
\begin{eqnarray*}
(\gamma _{A}\circ \gamma _{A})(a) &=&\dbigwedge\limits_{a=(ax)a}\left\{
\gamma _{A}(ax)\vee \gamma _{A}(a)\right\} \leq \gamma _{A}(ax)\vee \gamma
_{A}(a) \\
&\leq &\gamma _{A}(a)\vee \gamma _{A}(a)=(\gamma _{A}\cup \gamma _{A})(a).
\end{eqnarray*}

Thus we get that $\mu _{A}\circ \mu _{B}\supseteq \mu _{A}\cap \mu _{B}$ and 
$\gamma _{A}\circ \gamma _{B}\subseteq \gamma _{A}\cup \gamma _{B},$ which
give us $A\circ B\supseteq A\cap B$ and therefore $A\circ B=A\cap B.$
\end{proof}

\begin{example}
\label{e2}Let us consider an AG-groupoid $S=\left\{ a,b,c,d,e\right\} $ with
left identity $d$ in the following Cayley's table.
\end{example}

\begin{center}
\begin{tabular}{l|lllll}
. & $a$ & $b$ & $c$ & $d$ & $e$ \\ \hline
$a$ & $a$ & $a$ & $a$ & $a$ & $a$ \\ 
$b$ & $a$ & $e$ & $e$ & $c$ & $e$ \\ 
$c$ & $a$ & $e$ & $e$ & $b$ & $e$ \\ 
$d$ & $a$ & $b$ & $c$ & $d$ & $e$ \\ 
$e$ & $a$ & $e$ & $e$ & $e$ & $e$%
\end{tabular}
\end{center}

Note that $S$ is not regular, because $c\in S$ is not regular

The converse of Lemma \ref{fgh} is not true in general which is discussed in
the following.

Let us define an $IFS$ $A=(\mu _{A},\gamma _{A})$ of an AG-groupoid $S$ in
Example \ref{e2} as follows: $\mu _{A}(a)=\mu _{A}(b)=\mu _{A}(c)=0.3,$ $\mu
_{A}(d)=0.1$, $\mu _{A}(e)=0.4,$ $\gamma _{A}(a)=0.2,$ $\gamma _{A}(b)=0.3,$ 
$\gamma _{A}(c)=0.4,$ $\gamma _{A}(d)=0.5,$ $\gamma _{A}(e)=0.2.$ Then it is
easy to see that $A=(\mu _{A},\gamma _{A})$ is an intuitionistic fuzzy two
sided ideals of $S.$ Now again define an $IFS$ $B=(\mu _{B},\gamma _{B})$ of
an AG-groupoid $S$ in Example \ref{e2} as follows: $\mu _{B}(a)=\mu
_{B}(b)=\mu _{B}(c)=0.5,$ $\mu _{B}(d)=0.4$, $\mu _{B}(e)=0.6,$ $\gamma
_{B}(a)=0.3,$ $\gamma _{B}(b)=0.4,$ $\gamma _{B}(c)=0.5,$ $\gamma
_{B}(d)=0.6,$ $\gamma _{B}(e)=0.3.$ Then it is easy to observe that $A=(\mu
_{A},\gamma _{A})$ is an intuitionistic fuzzy two sided ideals of $S$ such
that $(\mu _{A}\circ \mu _{B})(a)=\{0.1,$ $0.3,$ $0.4\}=(\mu _{A}\cap \mu
_{B})(a)$ for all $a\in S$ and similarly $(\gamma _{A}\circ \gamma
_{B})(a)=(\gamma _{A}\cap \gamma _{B})$ for all $a\in S$, that is, $A\circ
B=A\cap B$ but $S$ is not regular$.$

An $IFS$ $A=(\mu _{A},\gamma _{A})$ of an AG-groupoid is said to be
idempotent if $\mu _{A}\circ \mu _{A}=\mu _{A}$ and $\gamma _{A}\circ \gamma
_{A}=\gamma _{A},$ that is, $A\circ A=A$ or $A^{2}=A.$

\begin{lemma}
\label{idem}Every intuitionistic fuzzy two-sided ideal $A=(\mu _{A},\gamma
_{A})$ of a regular AG-groupoid is idempotent.
\end{lemma}

\begin{proof}
Let $S$ be a regular AG-groupoid and let $A=(\mu _{A},\gamma _{A})$ be an
intuitionistic fuzzy two-sided ideal of $S.$ Now for $a\in S$ there exists $%
x\in S$ such that $a=(ax)a$ and therefore, we have 
\begin{eqnarray*}
(\mu _{A}\circ \mu _{A})(a) &=&\dbigvee\limits_{a=(ax)a}\{\mu _{A}(ax)\wedge
\mu _{A}(a)\}\geq \mu _{A}(ax)\wedge \mu _{A}(a) \\
&\geq &\mu _{A}(a)\wedge \mu _{A}(a)=\mu _{A}(a).
\end{eqnarray*}

Which shows that $\mu _{A}\circ \mu _{A}\supseteq \mu _{A}$ and by using
Lemma \ref{as}, $\mu _{A}\circ \mu _{A}\subseteq \mu _{A}$ and therefore $%
\mu _{A}\circ \mu _{A}=\mu _{A}.$ Similarly we can show that $\gamma
_{A}\circ \gamma _{A}=\gamma _{A},$ which shows that $A=(\mu _{A},\gamma
_{A})$ is idempotent.
\end{proof}

\begin{lemma}
\label{qw}In a regular AG-groupoid $S$, $A\circ \delta =A$ and $\delta \circ
A=A$ holds for every intuitionistic fuzzy two-sided ideal $A=(\mu
_{A},\gamma _{A})$ of $S,$ where $\delta =(S_{\delta },\Theta _{\delta }).$
\end{lemma}

\begin{proof}
Let $S$ be a regular AG-groupoid and let $A=(\mu _{A},\gamma _{A})$ be an
intuitionistic fuzzy two-sided ideal of $S$. Now for $a\in S$ there exists $%
x\in S$ such that $a=(ax)a,$ therefore 
\begin{eqnarray*}
(\mu _{A}\circ S_{\delta })(a) &=&\dbigvee\limits_{a=(ax)a}\left\{ \mu
_{A}(ax)\wedge S_{\delta }(a)\right\} \geq \mu _{A}(ax)\wedge S_{\delta }(a)
\\
&\geq &\mu _{A}(a)\wedge 1=\mu _{A}(a)
\end{eqnarray*}

and%
\begin{eqnarray*}
(\gamma _{A}\circ \Theta _{\delta })(a)
&=&\dbigwedge\limits_{a=(ax)a}\left\{ \gamma _{A}(ax)\wedge \Theta _{\delta
}(a)\right\} \leq \gamma _{A}(ax)\wedge \Theta _{\delta }(a) \\
&\leq &\gamma _{A}(a)\wedge 0=\gamma _{A}(a).
\end{eqnarray*}

Which shows that $\mu _{A}\circ S\supseteq \mu _{A}$ and $\gamma _{A}\circ
\Theta _{\delta }\subseteq \gamma _{A}.$ Now by using Lemma \ref{as}, we get 
$\mu _{A}\circ S=\mu _{A}$ and $\gamma _{A}\circ \Theta _{\delta }=\gamma
_{A}$. Therefore $A\circ \delta =A.$ Similarly we can prove that $S\circ
A=A. $
\end{proof}

\begin{corollary}
In a regular AG-groupoid $S$, $A\circ \delta =A$ and $\delta \circ A=A$ hold
for every intuitionistic fuzzy right ideal $A=(\mu _{A},\gamma _{A})$ of $S,$
where $\delta =(S_{\delta },\Theta _{\delta }).$
\end{corollary}

\begin{theorem}
The set of intuitionistic fuzzy two-sided ideals of a regular AG-groupoid $S$
forms a semilattice structure with identity $\delta ,$ where $\delta
=(S_{\delta },\Theta _{\delta }).$
\end{theorem}

\begin{proof}
Let $\mathbb{I}_{\mu \gamma }$ be the set of intuitionistic fuzzy two-sided
ideals of a regular AG-groupoid $S$ and let $A=(\mu _{A},\gamma _{A})$, $%
B=(\mu _{B},\gamma _{B})$ and $C=(\mu _{C},\gamma _{C})$ are any
intuitionistic fuzzy two sided ideals of $\mathbb{I}_{\mu \gamma }.$ Clearly 
$\mathbb{I}_{\mu \gamma }$ is closed and by Lemma \ref{idem}, we have $%
A=A^{2}$. Now by using Lemma \ref{fgh}, we get $A\circ B=B\circ A$ and
therefore, we have%
\begin{equation*}
(A\circ B)\circ C=(B\circ A)\circ C=(C\circ A)\circ B=(A\circ C)\circ
B=(B\circ C)\circ A=A\circ (B\circ C).
\end{equation*}

It is easy to see from Lemma \ref{qw} that $\delta $ is an identity in $%
\mathbb{I}_{\mu \gamma }.$
\end{proof}

\begin{lemma}
\label{qer}Every intuitionistic fuzzy right ideal of an AG-groupoid $S$ with
left identity is an intuitionistic fuzzy left ideal of $S$.
\end{lemma}

\begin{proof}
Let $S$ be an AG-groupoid with left identity and let $A=(\mu _{A},\gamma
_{A})$ be an intuitionistic fuzzy right ideal of $S$. Now by using $(1)$, we
have%
\begin{equation*}
\mu _{A}(ab)=\mu _{A}((ea)b)=\mu _{A}((ba)e)\geq \mu _{A}(b).
\end{equation*}

Similarly we can show that $\gamma _{A}(ab)\leq \gamma _{A}(b),$ which shows
that $A=(\mu _{A},\gamma _{A})$ is an intuitionistic fuzzy left ideal of $S.$
\end{proof}

The converse is not true in general because if we define an $IFS$ $A=(\mu
_{A},\gamma _{A})$ of an AG-groupoid $S$ in Example \ref{e2} as follows: $%
\mu _{A}(a)=0.8,\mu _{A}(b)=0.5,$ $\mu _{A}(c)=0.4,\mu _{A}(d)=0.3$ $\mu
_{A}(e)=0.6,$ $\gamma _{A}(a)=0.1,$ $\gamma _{A}(b)=0.7,$ $\gamma
_{A}(c)=0.6,$ $\gamma _{A}(d)=0.8$ and $\gamma _{A}(e)=0.3,$ then it is easy
to observe that $A=(\mu _{A},\gamma _{A})$ is an intuitionistic fuzzy left
ideal of $S$ but it is not an intuitionistic fuzzy right ideal of $S,$
because $\mu _{A}(bd)\ngeq \mu _{A}(b)$ and $\gamma _{A}(cd)\ngeq \gamma
_{A}(c).$

\begin{corollary}
\label{bg}Every intuitionistic fuzzy right ideal of a regular AG-groupoid $S$
with left identity is an intuitionistic fuzzy left ideal of $S.$
\end{corollary}

To consider the converse of Corollary \ref{bg}, we need to strengthen the
condition of a regular AG-groupoid $S$ which is given in the following.

An AG-groupoid $S$ is called a left duo if every left ideal of $S$ is a
two-sided ideal of $S.$

\begin{lemma}
\label{ll}Let $S$ be a regular AG-groupoid such that $S$ is a left duo, then
every intuitionistic fuzzy left ideal of $S$ is an intuitionistic fuzzy
right ideal of $S$.
\end{lemma}

\begin{proof}
Let $S$ be a left duo regular AG-groupoid and let $A=(\mu _{A},\gamma _{A})$
be an intuitionistic fuzzy left ideal of $S$. Let $x,y\in S$ then the left
ideal $Sx$ of $S$ is a two sided ideal of $S$ and since $S$ is regular
therefore by using $(1)$, we have 
\begin{equation*}
xy\in ((xS)x)y\subseteq ((xS)((xS)x))S\subseteq ((((xS)x)S)x)S\subseteq
(Sx)S\subseteq S.
\end{equation*}

It follows that there exists $w\in S$ such that $xy=wx.$ As $A=(\mu
_{A},\gamma _{A})$ is an intuitionistic fuzzy left ideal of $S$, therefore
we get $\mu _{A}(xy)=\mu _{A}(wx)\geq \mu _{A}(x)$ and $\mu _{A}(xy)=\gamma
_{A}(wx)\leq \gamma _{A}(x).$ This means that $A=(\mu _{A},\gamma _{A})$ is
an intuitionistic fuzzy right ideal of $S.$
\end{proof}

An AG-groupoid $S$ is called an intuitionistic fuzzy left duo if every
intuitionistic fuzzy left ideal of $S$ is an intuitionistic fuzzy two-sided
ideal of $S$.

\begin{corollary}
\label{dh}Let $S$ be a regular AG-groupoid. If $S$ is a left duo, then $S$
is an intuitionistic fuzzy left duo.
\end{corollary}

\begin{theorem}
\label{c1}If $A=(\mu _{A},\gamma _{A})$ is an intuitionistic fuzzy two-sided
ideal of a regular AG-groupoid $S$ with left identity, then $A(ab)=A(ba)$
holds for all $a,b$ in $S$.
\end{theorem}

\begin{proof}
Let $A=(\mu _{A},\gamma _{A})$ be an intuitionistic fuzzy two-sided ideal of
a regular AG-groupoid $S$ with left identity and let $a,b\in S$, then $%
a=(ax)a$ and $b=(by)b$ for some $x,y\in S.$ Now by using $(2)$ and $(3),$ we
have%
\begin{eqnarray*}
\mu _{A}(ab) &=&\mu _{A}(((ax)a)((by)b))=\mu _{A}(((ax)(by))(ab))=\mu
_{A}((ba)((by)(ax))) \\
&\geq &\mu _{A}(ba)=\mu _{A}(((by)b)((ax)a))=\mu _{A}(((by)(ax))(ba)) \\
&=&\mu _{A}((ab)((ax)(by)))\geq \mu _{A}(ab).
\end{eqnarray*}

Which shows that $\mu _{A}(ab)=\mu _{A}(ba)$ holds for all $a,b$ in $S$ and
similarly $\gamma _{A}(ab)=\gamma _{A}(ba)$ holds for all $a,b$ in $S.$ Thus 
$A(ab)=A(ba)$ holds for all $a,b$ in $S.$
\end{proof}

The converse is not true in general. For this consider an $IFS$ $A=(\mu
_{A},\gamma _{A})$ of a regular AG-groupoid $S$ considered in Example \ref%
{ex} as follows: $\mu _{A}(a)=0.1,$ $\mu _{A}(b)=0.2,$ $\mu _{A}(c)=0.6,$ $%
\mu _{A}(d)=0.4,$ $\mu _{A}(e)=0.6,$ $\gamma _{A}(a)=0.2,$ $\gamma
_{A}(b)=0.3,$ $\gamma _{A}(c)=0.7,$ $\gamma _{A}(d)=0.5,$ $\gamma
_{A}(e)=0.7,$ then it is easy to observe that $A(ab)=A(ba)$ holds for all $a$
and $b$ in $S$ but $A=(\mu _{A},\gamma _{A})$ is not an intuitionistic fuzzy
two-sided ideal of $S,$ because $\mu _{A}(cc)\ngeqslant \mu _{A}(c)$ and $%
\gamma _{A}(ed)\nleqslant \gamma (d)$ $(\gamma _{A}(de)\nleqslant \gamma
(d)).$

\begin{corollary}
\label{jki}If $A=(\mu _{A},\gamma _{A})$ is an intuitionistic fuzzy right
ideal of a regular AG-groupoid $S$ with left identity, then $A(ab)=A(ba)$
holds for all $a,b$ in $S$.
\end{corollary}

The converse of Corollary \ref{jki} is not true in general which can be
followed from the converse of Theorem \ref{c1}.

\begin{theorem}
\label{ki}Let $S$ be a regular AG-groupoid with left identity, then $A=(\mu
_{A},\gamma _{A})$ is an intuitionistic fuzzy left ideal of $S$ if and only
if $A=(\mu _{A},\gamma _{A})$ is an intuitionistic fuzzy bi-ideal of $S.$
\end{theorem}

\begin{proof}
Let $A=(\mu _{A},\gamma _{A})$ be an intuitionistic fuzzy left ideal of a
regular AG-groupoid $S$ and let $w,x,y\in S,$ then by using $(1)$,we have%
\begin{equation*}
\mu _{A}((xw)y)=\mu _{A}(((yw)x))\geq \mu _{A}(x)\geq \mu _{A}(x)\wedge \mu
_{A}(y)
\end{equation*}

and%
\begin{equation*}
\gamma _{A}((xw)y)=\gamma _{A}(((yw)x))\leq \mu _{A}(x)\leq \mu _{A}(x)\vee
\mu _{A}(y).
\end{equation*}

Thus $A=(\mu _{A},\gamma _{A})$ is an intuitionistic fuzzy bi-ideal of $S.$

Conversely, let $A=(\mu _{A},\gamma _{A})$ be an intuitionistic fuzzy
bi-ideal of $S$ and let $x,y\in S,$ then there exists $z\in S$ such that $%
y=(yz)y.$ Now by using $(4),(2),$ $(1)$ and $(3),$ we have 
\begin{eqnarray*}
\mu _{A}(xy) &=&\mu _{A}(x((yz)y))=\mu _{A}((yz)(xy))=\mu _{A}((yx)(zy))=\mu
_{A}(((zy)x)y) \\
&=&\mu _{A}(((zy)(ex))y)=\mu _{A}(((xe)(yz))y)=\mu _{A}((y((xe)z))y)\geq \mu
_{A}(y).
\end{eqnarray*}

Similarly $\gamma _{A}(xy)\leq \gamma _{A}(y)$ and therefore $A=(\mu
_{A},\gamma _{A})$ is an intuitionistic fuzzy left ideal of $S.$
\end{proof}

\begin{corollary}
Let $S$ be a regular AG-groupoid with left identity, then $A=(\mu
_{A},\gamma _{A})$ is an intuitionistic fuzzy left ideal of $S$ if and only
if $A=(\mu _{A},\gamma _{A})$ is an intuitionistic fuzzy generalized
bi-ideal of $S.$
\end{corollary}

\begin{theorem}
Let $S$ be a regular AG-groupoid with left identity such that $S$ is a left
duo, then $A=(\mu _{A},\gamma _{A})$ is an intuitionistic fuzzy right ideal
of $S$ if and only if $A=(\mu _{A},\gamma _{A})$ is an intuitionistic fuzzy
bi-ideal of $S.$
\end{theorem}

\begin{proof}
It follows from Theorem \ref{ki} and Lemma \ref{ll}.
\end{proof}

\begin{corollary}
Let $S$ be a regular AG-groupoid with left identity such that $S$ is a left
duo, then $A=(\mu _{A},\gamma _{A})$ is an intuitionistic fuzzy right ideal
of $S$ if and only if $A=(\mu _{A},\gamma _{A})$ is an intuitionistic fuzzy
generalized bi-ideal of $S.$
\end{corollary}

\begin{theorem}
\label{asdf}Let $S$ be a regular AG-groupoid with left identity, then $%
A=(\mu _{A},\gamma _{A})$ is an intuitionistic fuzzy $(1,2)$-ideal of $S$ if 
$A=(\mu _{A},\gamma _{A})$ is an intuitionistic fuzzy left ideal of $S.$
\end{theorem}

\begin{proof}
Let $A=(\mu _{A},\gamma _{A})$ be an intuitionistic fuzzy left ideal of a
regular AG-groupoid $S$ and let $w,x,y,z\in S,$ then there exists $a,b\in S$
such that $x=(xa)x$ and $y=(yb)y.$ Now by using $(1)$ and $(4)$, we have%
\begin{eqnarray*}
\mu _{A}((xw)(yz)) &=&\mu _{A}((((xa)x)w)(yz))=\mu _{A}(((wx)(xa))(yz)) \\
&=&\mu _{A}((x((wx)a))(yz))=\mu _{A}(((yz)((wx)a))x)\geq \mu _{A}(x).
\end{eqnarray*}

Now by using $(4),(2)$ and $(1),$ we have%
\begin{eqnarray*}
\mu _{A}((xw)(yz)) &=&\mu _{A}(y((wx)z))=\mu _{A}(((yb)y)((wx)z))=\mu
_{A}(((yb)(wx))(yz)) \\
&=&\mu _{A}(((((yb)y)b)(wx))(yz))=\mu _{A}((((by)(yb))(wx))(yz)) \\
&=&\mu _{A}(((y((by)b))(wx))(yz))=\mu _{A}((((wx)((by)b))y)(yz)) \\
&=&\mu _{A}((((by)((wx)b))y)(yz))=\mu _{A}(((y((wx)b))(yb))(yz)) \\
&=&\mu _{A}((y((y((wx)b))b))(yz))=\mu _{A}(((yz)(y(((wx)b)b)))y)\geq \mu
_{A}(y).
\end{eqnarray*}

Now by using $(3)$ and $(1)$, we have%
\begin{equation*}
\mu _{A}((xw)(yz))=\mu _{A}((zy)(wx))=\mu _{A}(((wx)y)z)\geq \mu _{A}(z).
\end{equation*}

Thus we get, $\mu _{A}((xw)(yz))\geq \mu _{A}(x)\wedge \mu _{A}(y)\wedge \mu
_{A}(z)$ and similarly $\gamma _{A}((xw)(yz))\leq \gamma _{A}(x)\vee \gamma
_{A}(y)\vee \gamma _{A}(z).$ Thus $A=(\mu _{A},\gamma _{A})$ is an
intuitionistic fuzzy $(1,2)$-ideal of $S.$
\end{proof}

For the converse of Theorem \ref{asdf}, we have to strengthen the condition
of a regular AG-groupoid which is given in the following.

An AG-groupoid $S$ is called an AG-band if $a=a^{2}$ for all $a\in S.$

\begin{theorem}
Let $S$ be a regular AG-band, then $A=(\mu _{A},\gamma _{A})$ is an
intuitionistic fuzzy left ideal of $S$ if $A=(\mu _{A},\gamma _{A})$ is an
intuitionistic fuzzy $(1,2)$-ideal of $S.$
\end{theorem}

\begin{proof}
It is simple.
\end{proof}

\begin{theorem}
Let $S$ be a regular AG-groupoid with left identity such that $S$ is a left
duo, then $A=(\mu _{A},\gamma _{A})$ is an intuitionistic fuzzy $(1,2)$%
-ideal of $S$ if $A=(\mu _{A},\gamma _{A})$ is an intuitionistic fuzzy right
ideal of $S.$
\end{theorem}

\begin{proof}
It is an easy consequence of Theorem \ref{asdf} and Lemma \ref{ll}.
\end{proof}

\begin{theorem}
Let $S$ be a regular AG-band, then $A=(\mu _{A},\gamma _{A})$ is an
intuitionistic fuzzy right ideal of $S$ if $A=(\mu _{A},\gamma _{A})$ is an
intuitionistic fuzzy $(1,2)$-ideal of $S.$
\end{theorem}

\begin{proof}
It is simple.
\end{proof}

\begin{lemma}
$($\cite{Mordeson},\cite{mad}$)$ \label{00} For any IFS $A=(\mu _{A},\gamma
_{A})$ of an AG-groupoid $S,$ the following properties holds.
\end{lemma}

$(i)$ $A$ is an AG-subgroupoid of $S$ if and only if $\chi _{A}$ is an
intuitionistic fuzzy AG-subgroupoid of $S$.

$(ii)$ $A$ is an intuitionistic left (right, two-sided) ideal of $S$ if and
only if $\chi _{A}$ is an intuitionistic fuzzy left (right, two-sided) ideal
of $S$.

A subset $A$ of an AG-groupoid $S$ is called semiprime if $a^{2}\in A$
implies $a\in A.$ An $IFS$ $A=(\mu _{A},\gamma _{A})$ of an AG-groupoid $S$
is called an intuitionistic fuzzy semiprime if $\mu _{A}(a)\geq \mu
_{A}(a^{2})$ and $\gamma _{A}(a)\leq \gamma _{A}(a^{2})$ for all $a$ in $S.$

\begin{lemma}
\label{1}Every right (left, two-sided ideal) of an AG-groupoid $S$ is
semiprime if and only if their characteristic functions are intuitionistic
fuzzy semiprime.
\end{lemma}

\begin{proof}
Let $R$ be any right ideal of an AG-groupoid $S$, then by Lemma \ref{00},
the intuitionistic characteristic function of $R,$ that is, $\chi _{R}=(\mu
_{\chi _{R}},\gamma _{\chi _{R}})$ is an intuitionistic fuzzy right ideal of 
$S$. Let $a^{2}\in R$, then $\mu _{\chi _{R}}(a^{2})=1$ and assume that $R$
is semiprime, then $a\in R,$ which implies that $\mu _{\chi _{R}}(a)=1$.
Thus we get $\mu _{\chi _{R}}(a^{2})=\mu _{\chi _{R}}(a)$ and similarly we
can show that $\gamma _{\chi _{R}}(a^{2})=\gamma _{\chi _{R}}(a)$, therefore 
$\chi _{R}=(\mu _{\chi _{R}},\gamma _{\chi _{R}})$ is an intuitionistic
fuzzy semiprime. The converse is simple.
\end{proof}

\begin{corollary}
\label{11}Let $S$ be an AG-groupoid, then every right (left, two-sided)
ideal of $S$ is semiprime if every intuitionistic fuzzy right (left,
two-sided) ideal of $S$ is an intuitionistic fuzzy semiprime.
\end{corollary}

The converse is not true in general. For this let us consider an AG-groupoid 
$S$ in Example \ref{e2}. It is easy to observe that the only left ideals of $%
S$ are $\{a,b,e\},$ $\{a,c,e\},$ $\{a,b,c,e\}$ and $\{a,e\}$ which are
semiprime. Clearly the right and two sided ideals of $S$ are $\{a,b,c,e\}$
and $\{a,e\}$ which are also semiprime. Now on the other hand, if we define
an $IFS$ $A=(\mu _{A},\gamma _{A})$ of $S$ as follows: $\mu _{A}(a)=\mu
_{A}(b)=\mu _{A}(c)=0.2,$ $\mu _{A}(d)=0.1$, $\mu _{A}(e)=0.3,$ $\gamma
_{A}(a)=0.2,$ $\gamma _{A}(b)=\gamma _{A}(c)=0.5,$ $\gamma _{A}(d)=0.6$ and $%
\gamma _{A}(e)=0.3,$ then $A=(\mu _{A},\gamma _{A})$ is a fuzzy right (left,
two-sided) ideal of $S$ but $A=(\mu _{A},\gamma _{A})$ is not an
intuitionistic fuzzy semiprime, because $\mu _{A}(c)\ngeq \mu _{A}(c^{2})$
and $\gamma _{A}(c)\nleqslant \gamma _{A}(c^{2}).$

An element $a$ of an AG-groupoid $S$ is called an intra-regular if there
exist $x,y\in S$ such that $a=(xa^{2})y$ and $S$ is called an intra-regular
if every element of $S$ is an intra-regular.

\begin{example}
\label{t}Let $S=\{a,b,c,d,e\}$ be an AG-groupoid with left identity $b$ in
the following cayley's table.
\end{example}

\begin{center}
\begin{tabular}{l|lllll}
. & $a$ & $b$ & $c$ & $d$ & $e$ \\ \hline
$a$ & $a$ & $a$ & $a$ & $a$ & $a$ \\ 
$b$ & $a$ & $b$ & $c$ & $d$ & $e$ \\ 
$c$ & $a$ & $e$ & $b$ & $c$ & $d$ \\ 
$d$ & $a$ & $d$ & $e$ & $b$ & $c$ \\ 
$e$ & $a$ & $c$ & $d$ & $e$ & $b$%
\end{tabular}
\end{center}

\QTP{Body Math}
Clearly $S$ is an intra-regular because, $a=(aa^{2})a,$ $b=(cb^{2})e,$ $%
c=(dc^{2})e,$ $d=(cd^{2})c$ and $e=(be^{2})e.$

\begin{lemma}
\label{2}For an intra-regular AG-groupoid $S$ with left identity, the
following holds.
\end{lemma}

$(i)$ Every intuitionistic fuzzy right ideal of $S$ is an intuitionistic
fuzzy semiprime.

$(ii)$ Every intuitionistic fuzzy left ideal of $S$ is an intuitionistic
fuzzy semiprime.

$(iii)$ Every intuitionistic fuzzy two-sided ideal of $S$ is an
intuitionistic fuzzy semiprime.

\begin{proof}
$(i):$ Let $A=(\mu _{A},\gamma _{A})$ be an intuitionistic fuzzy right ideal
of an intra-regular AG-groupoid $S$ with left identity and \ let $a\in S,$
then there exists $x,y\in S$ such that $a=(xa^{2})y.$ Now by using $(3)$ and 
$(4),$ we have 
\begin{equation*}
\mu _{A}(a)=\mu _{A}((xa^{2})y)=\mu _{A}((xa^{2})(ey))=\mu
_{A}((ye)(a^{2}x))=\mu _{A}(a^{2}((ye)x))\geq \mu _{A}(a^{2})
\end{equation*}

and similarly%
\begin{equation*}
\gamma _{A}(a)=\gamma _{A}((xa^{2})y)=\gamma _{A}((xa^{2})(ey))=\gamma
_{A}((ye)(a^{2}x))=\gamma _{A}(a^{2}((ye)x))\leq \gamma _{A}(a^{2}).
\end{equation*}

Thus $A=(\mu _{A},\gamma _{A})$ is an intuitionistic fuzzy semiprime.

$(ii):$ Let $A=(\mu _{A},\gamma _{A})$ be an intuitionistic fuzzy left ideal
of an intra-regular AG-groupoid $S$ with left identity and let $a\in S,$
then there exist $x,y\in S$ such that $a=(xa^{2})y.$ Now by using $(4),$ $%
(3) $ and $(1),$ we have%
\begin{eqnarray*}
\mu _{A}(a) &=&\mu _{A}((xa^{2})y)=\mu _{A}((x(aa))y)=\mu _{A}((a(xa))y)=\mu
_{A}((((xa^{2})y)(xa))y) \\
&=&\mu _{A}(((ax)(y(xa^{2})))y)=\mu _{A}(((ax)(y((ex)(aa))))y) \\
&=&\mu _{A}(((ax)(y(a^{2}(xe))))y)=\mu _{A}(((ax)(a^{2}(y(xe))))y) \\
&=&\mu _{A}(a^{2}((ax)(y(xe)))y)=\mu _{A}((y((y(xe))(ax)))a^{2})\geq \mu
_{A}(a^{2}).
\end{eqnarray*}

Similarly we can show that $\gamma _{A}(a)\leq \gamma _{A}(a^{2})$ and
therefore $A=(\mu _{A},\gamma _{A})$ is an intuitionistic fuzzy semiprime.

$(iii):$ It can be followed from $(i)$ and $(ii).$
\end{proof}

\begin{theorem}
\label{3}Let $S$ be an AG-groupoid with left identity, then the following
statements are equivalent.
\end{theorem}

$(i)$ $S$ is intra-regular.

$(ii)$ Every intuitionistic fuzzy right (left, two-sided) ideal of $S$ is an
intuitionistic fuzzy semiprime.

\begin{proof}
$(i)\Longrightarrow (ii)$ is followed by Lemma \ref{2}.

$(ii)\Longrightarrow (i):$ Let $S$ be an AG-groupoid with left identity and
let every intuitionistic fuzzy right (left, two-sided) ideal of $S$ is an
intuitionistic fuzzy semiprime. Since $a^{2}S$ is a right and also a left
ideal of $S$, therefore by using Corollary \ref{11}, $a^{2}S\ $is semiprime.
Clearly $a^{2}\in a^{2}S$ and therefore $a\in a^{2}S$. Now by using $(1)$,
we have%
\begin{equation*}
a\in a^{2}S=(aa)S=(Sa)a\subseteq
(Sa)(a^{2}S)=((a^{2}S)a)S=((aS)a^{2})S\subseteq (Sa^{2})S.
\end{equation*}%
Which shows that $S$ is an intra-regular.
\end{proof}

\begin{theorem}
The following statements are equivalent for an AG-groupoid with left
identity.
\end{theorem}

$(i)$ $S\,$is an intra-regular.

$(ii)$ Every intuitionistic fuzzy right ideal of $S$ is an intuitionistic
fuzzy semiprime.

$(iii)$ Every intuitionistic fuzzy left ideal of $S$ is an intuitionistic
fuzzy semiprime.

\begin{proof}
$(i)\Longrightarrow (iii)$ and $(ii)\Longrightarrow (i)$ are followed by
Theorem \ref{3}.

$(iii)\Longrightarrow (ii):$ Let $S$ be an AG-groupoid with left identity
and let every intuitionistic fuzzy left ideal of $S$ is an intuitionistic
fuzzy semiprime. Let $A=(\mu _{A},\gamma _{A})$ be an intuitionistic fuzzy
right ideal of $S$. Now by using Lemma \ref{qer}, $A=(\mu _{A},\gamma _{A})$
is an intuitionistic
\end{proof}

\begin{lemma}
$($\cite{Mordeson},\cite{mad}$)$ \label{000} For any non-empty $IFSs$ $%
A=(\mu _{A},\gamma _{A})$ and $B=(\mu _{B},\gamma _{B})$ of an AG-groupoid $%
S $, then $\chi _{A}\circ \chi _{B}=\chi _{AB}.$
\end{lemma}

\begin{theorem}
For an AG-groupoid $S$ with left identity, the following conditions are
equivalent.
\end{theorem}

$(i)$ $S$ is an intra-regular.

$(ii)$ $R\cap L=RL,$ $R$ is any right ideal and $L$ is any left ideal of $S$
such that $R$ is semiprime.

$(iii)$ $A\cap B=A\circ B,$ $A=(\mu _{A},\gamma _{A})$ is any intuitionistic
fuzzy right ideal and $B=(\mu _{B},\gamma _{B})$ is any intuitionistic fuzzy
left ideal of $S$ such that $A=(\mu _{A},\gamma _{A})$ is an intuitionistic
fuzzy semiprime.

\begin{proof}
$(i)\Longrightarrow (iii):$ Assume that $S$ is an intra-regular AG-groupoid.
Let $A=(\mu _{A},\gamma _{A})$ is any intuitionistic fuzzy right ideal and $%
B=(\mu _{B},\gamma _{B})$ is any intuitionistic fuzzy left ideal of $S.$ Now
for $a\in S$ there exists $x,y\in S$ such that $a=(xa^{2})y.$ Now by using $%
(4),$ $(1)$ and $(3),$ we have 
\begin{eqnarray*}
a &=&(x(aa))y=(a(xa))y=(y(xa))a=(y(x((xa^{2})y)))a=(y((xa^{2})(xy)))a \\
&=&(y((yx)(a^{2}x)))a=(y(a^{2}((yx)x)))a=(a^{2}(y((yx)x)))a.
\end{eqnarray*}

Therefore 
\begin{eqnarray*}
(\mu _{A}\circ \mu _{B})(a) &=&\dbigvee\limits_{a=(a^{2}(y((yx)x)))a}\{\mu
_{A}(a^{2}(y((yx)x)))\wedge \mu _{B}(a)\} \\
&\geq &\mu _{A}(a)\wedge \mu _{B}(a)=(\mu _{A}\cap \mu _{B})(a)
\end{eqnarray*}

and%
\begin{eqnarray*}
(\gamma _{A}\circ \gamma _{B})(a)
&=&\dbigwedge\limits_{a=(a^{2}(y((yx)x)))a}\left\{ \gamma
_{A}(a^{2}(y((yx)x)))\vee \gamma _{B}(a)\right\} \\
&\leq &\gamma _{A}(a)\wedge \gamma _{B}(a)=(\gamma _{A}\cup \gamma _{B})(a).
\end{eqnarray*}

Which implies that $A\circ B\supseteq A\cap B$ and by using Lemma \ref{as}, $%
A\circ B\subseteq A\cap B,$ therefore $A\cap B=A\circ B$.

$(iii)\Longrightarrow (ii)$

Let $R$ be any right ideal and $L$ be any left ideal of an AG-groupoid $S,$
then by Lemma \ref{00}, $\chi _{R}=(\mu _{\chi _{R}},\gamma _{\chi _{R}})$
and $\chi _{L}=(\mu _{\chi _{L}},\gamma _{\chi _{L}})$ are an intuitionistic
fuzzy right and intuitionistic fuzzy left ideals of $S$ respectively. As $%
RL\subseteq R\cap L$ is obvious therefore let $a\in R\cap L,$ then $a\in R$
and $a\in L.$ Now by using Lemma \ref{000} and given assumption, we have%
\begin{equation*}
\mu _{\chi _{RL}}(a)=(\mu _{\chi _{R}}\circ \mu _{\chi _{L}})(a)=(\mu _{\chi
_{R}}\cap \mu _{\chi _{L}})(a)=\mu _{\chi _{R}}(a)\wedge \mu _{\chi
_{L}}(a)=1
\end{equation*}

and similarly%
\begin{equation*}
\gamma _{\chi _{RL}}(a)=(\gamma _{\chi _{R}}\circ \gamma _{\chi
_{L}})(a)=(\gamma _{\chi _{R}}\cup \gamma _{\chi _{L}})(a)=\gamma _{\chi
_{R}}(a)\vee \gamma _{\chi _{L}}(a)=1.
\end{equation*}

Which implies that $a\in RL$ and therefore $R\cap L=RL.$ Now by using
Corollary \ref{11}, $R$ is semiprime.

$(ii)\Longrightarrow (i)$

Let $S$ be an AG-groupoid, then clearly $Sa$ is a left ideal of $S$ such
that $a\in Sa$ and $a^{2}S$ is a right ideal of $S$ such that $a^{2}\in
a^{2}S.$ Since by assumption, $a^{2}S$ is semiprime, therefore $a\in a^{2}S.$
Now by using $(3),$ $(1)$ and $(4)$, we have 
\begin{eqnarray*}
a &\in &a^{2}S\cap Sa=(a^{2}S)(Sa)=(aS)(Sa^{2})=((Sa^{2})S)a=((Sa^{2})(SS))a
\\
&=&((SS)(a^{2}S))a=(a^{2}((SS)S))a\subseteq (a^{2}S)S=(SS)(aa)=a^{2}S \\
&=&(aa)S=(Sa)a\subseteq (Sa)(a^{2}S)=((a^{2}S)a)S=((aS)a^{2})S\subseteq
(Sa^{2})S.
\end{eqnarray*}

Which shows that $S$ is an intra-regular.
\end{proof}

\begin{lemma}
\label{iff}Let $A=(\mu _{A},\gamma _{A})$ be an $IFS$ of an intra-regular
AG-groupoid $S$ with left identity$,$ then $A=(\mu _{A},\gamma _{A})$ is an
intuitionistic fuzzy left ideal of $S$ if and only if $A=(\mu _{A},\gamma
_{A})$ is an intuitionistic fuzzy right ideal of $S.$
\end{lemma}

\begin{proof}
Let $S$ be an intra-regular AG-groupoid and let $A=(\mu _{A},\gamma _{A})$
be an intuitionistic fuzzy left ideal of $S.$ Now for $a,b\in S$ there
exists $x,y,x^{^{\prime }},y^{^{\prime }}\in S$ such that $a=(xa^{2})y$ and $%
b=(x^{^{\prime }}b^{2})y^{^{\prime }},$ then by using $(1),$ $(3)$ and $(4),$
we have 
\begin{eqnarray*}
\mu _{A}(ab) &=&\mu _{A}(((xa^{2})y)b)=\mu _{A}((by)(x(aa)))=\mu
_{A}(((aa)x)(yb)) \\
&=&\mu _{A}(((xa)a)(yb))=\mu _{A}(((xa)(ea))(yb))=\mu _{A}(((ae)(ax))(yb)) \\
&=&\mu _{A}((a((ae)x))(yb))=\mu _{A}(((yb)((ae)x))a)\geq \mu _{A}(a).
\end{eqnarray*}

Similarly we can get $\gamma _{A}(ab)\leq \gamma _{A}(a),$ which implies
that $A=(\mu _{A},\gamma _{A})$ is an intuitionistic fuzzy right ideal of $%
S. $

Conversely let $A=(\mu _{A},\gamma _{A})$ be an intuitionistic fuzzy right
ideal of $S.$ Now by using $(4)$ and $(3),$ we have 
\begin{eqnarray*}
\mu _{A}(ab) &=&\mu _{A}(a((x^{^{\prime }}b^{2})y^{^{\prime }})=\mu
_{A}((x^{^{\prime }}b^{2})(ay^{^{\prime }}))=\mu _{A}((y^{^{\prime
}}a)(b^{2}x^{^{\prime }})) \\
&=&\mu _{A}(b^{2}((y^{^{\prime }}a)x))\geq \mu _{A}(b).
\end{eqnarray*}

Similarly we can get $\gamma _{A}(ab)\leq \gamma _{A}(b),$ which implies
that $A=(\mu _{A},\gamma _{A})$ is an intuitionistic fuzzy left ideal of $S.$
\end{proof}

\begin{theorem}
Let $S$ be an intra-regular AG-groupoid with left identity and let $A=(\mu
_{A},\gamma _{A})$ be an IFS, then the following conditions are equivalent.
\end{theorem}

$(i)$ $A=(\mu _{A},\gamma _{A})$ is an intuitionistic fuzzy two-sided ideal
of $S$.

$(ii)$ $A=(\mu _{A},\gamma _{A})$ is an intuitionistic fuzzy bi-ideal of $S$.

\begin{proof}
$(i)\Longrightarrow (ii)$ is simple.

$(ii)\Longrightarrow (i):$ Let $A=(\mu _{A},\gamma _{A})$ be an
intuitionistic fuzzy generalized bi-ideal of an intra-regular AG-groupoid $S$
and let $a\in S$, then there exists $x,y\in S$ such that $a=(xa^{2})y.$ Now
by using $(4),$ $(1)$ and $(3),$ we have%
\begin{eqnarray*}
\mu _{A}(ab) &=&\mu _{A}(((x(aa))y)b)=\mu _{A}(((a(xa))y)b)=\mu
_{A}((by)((ea)(xa))) \\
&=&\mu _{A}((by)((ax)(ae)))=\mu _{A}(((ae)(ax))(yb))=\mu _{A}((a((ae)x))(yb))
\\
&=&\mu _{A}(((yb)((ae)x))a)=\mu _{A}(((yb)((((xa^{2})y)e)x))a) \\
&=&\mu _{A}(((yb)((y(xa^{2}))(ex)))a)=\mu _{A}(((yb)((xe)((xa^{2})(ey))))a)
\\
&=&\mu _{A}(((yb)((xe)((ye)(a^{2}x))))a)=\mu
_{A}(((yb)((xe)(a^{2}((ye)x))))a) \\
&=&\mu _{A}(((yb)(a^{2}((xe)((ye)x))))a)=\mu
_{A}((a^{2}((yb)((xe)((ye)x))))a) \\
&\geq &\mu _{A}(a^{2})\wedge \mu _{A}(a)\geq \mu _{A}(a)\wedge \mu
_{A}(a)\wedge \mu _{A}(a)=\mu _{A}(a).
\end{eqnarray*}

Similarly we can show that $\gamma _{A}(ab)\leq \gamma _{A}(a)$ and
therefore $A=(\mu _{A},\gamma _{A})$ is an intuitionistic fuzzy right ideal
of $S$. Now by using Lemma \ref{iff}, $A=(\mu _{A},\gamma _{A})$ is an
intuitionistic fuzzy two-sided ideal of $S$.
\end{proof}

\begin{theorem}
Let $S$ be an intra-regular AG-groupoid with left identity and let $A=(\mu
_{A},\gamma _{A})$ be an IFS, then the following conditions are equivalent.
\end{theorem}

$(i)$ $A=(\mu _{A},\gamma _{A})$ is an intuitionistic fuzzy two-sided ideal
of $S$.

$(ii)$ $A=(\mu _{A},\gamma _{A})$ is an intuitionistic fuzzy $(1,2)$-ideal
of $S$.

\begin{proof}
$(i)\Longrightarrow (ii)$ is simple.

$(ii)\Longrightarrow (i):$ Let $A=(\mu _{A},\gamma _{A})$ be an
intuitionistic fuzzy generalized bi-ideal of an intra-regular AG-groupoid $S$
and let $b\in S$, then there exists $x,y\in S$ such that $b=(xb^{2})y.$ Now
by using $(4),$ $(1)$ and $(3),$ we have%
\begin{eqnarray*}
\mu _{A}(ab) &=&\mu _{A}(a((xb^{2})y))=\mu _{A}((x(bb))(ay))=\mu
_{A}((b(xb))(ay)) \\
&=&\mu _{A}(((ay)(xb))b)=\mu _{A}(((ay)(xb))((xb^{2})y))=\mu
_{A}((xb^{2})(((ay)(xb))y)) \\
&=&\mu _{A}((y((ay)(xb)))(b^{2}x))=\mu _{A}(b^{2}((y((ay)(xb)))x)) \\
&=&\mu _{A}((bb)((y((ay)(xb)))x))=\mu _{A}((x(y((ay)(xb))))(bb)) \\
&=&\mu _{A}((x(y((bx)(ya))))(bb))=\mu _{A}((x((bx)(y(ya))))(bb)) \\
&=&\mu _{A}(((bx)(x(y(ya))))(bb))=\mu _{A}(((((xb^{2})y)x)(x(y(ya))))(bb)) \\
&=&\mu _{A}((((xy)(xb^{2}))(x(y(ya))))(bb))=\mu
_{A}((((b^{2}x)(yx))(x(y(ya))))(bb)) \\
&=&\mu _{A}(((((yx)x)b^{2})(x(y(ya))))(bb))=\mu
_{A}((((y(ya))x)(b^{2}((yx)x)))(bb)) \\
&=&\mu _{A}((((y(ya))x)(b^{2}(x^{2}y)))(bb))=\mu
_{A}((b^{2}(((y(ya))x)(x^{2}y)))(bb)) \\
&=&\mu _{A}((((x^{2}y)((y(ya))x))(bb))(bb))=\mu
_{A}((a((x^{2}y)(((y(ya))x)b)))(bb)) \\
&\geq &\mu _{A}(b)\wedge \mu _{A}(b)\wedge \mu _{A}(b)=\mu _{A}(b).
\end{eqnarray*}

Similarly we can show that $\gamma _{A}(ab)\leq \gamma _{A}(b)$ and
therefore $A=(\mu _{A},\gamma _{A})$ is an intuitionistic fuzzy left ideal
of $S$. Now by using Lemma \ref{iff}, $A=(\mu _{A},\gamma _{A})$ is an
intuitionistic fuzzy two-sided ideal of $S$.
\end{proof}

\begin{theorem}
Let $S$ be an intra-regular AG-groupoid with left identity and let $A=(\mu
_{A},\gamma _{A})$ be an IFS, then the following conditions are equivalent.
\end{theorem}

$(i)$ $A=(\mu _{A},\gamma _{A})$ is an intuitionistic fuzzy bi-ideal of $S$.

$(ii)$ $A=(\mu _{A},\gamma _{A})$ is an intuitionistic fuzzy generalized
bi-ideal of $S$.

\begin{proof}
$(i)\Longrightarrow (ii)$ is obvious.

$(ii)\Longrightarrow (i):$ Let $A=(\mu _{A},\gamma _{A})$ be an
intuitionistic fuzzy generalized bi-ideal of an intra-regular AG-groupoid $S$
and let $a\in S$, then there exists $x,y\in S$ such that $a=(xa^{2})y.$ Now
by using $(4),$ $(3)$ and $(1),$ we have 
\begin{eqnarray*}
\mu _{A}(ab) &=&\mu _{A}(((x(aa))y)b)=\mu _{A}((((ea)(xa))y)b)=\mu
_{A}((((ax)(ae))y)b) \\
&=&\mu _{A}(((a((ax)e))(ey))b)=\mu _{A}(((ye)(((ax)e)a))b) \\
&=&\mu _{A}(((ye)((ae)(ax)))b)=\mu _{A}(((ye)(a((ae)x)))b) \\
&=&\mu _{A}((a((ye)((ae)x)))b)\geq \mu _{A}(a)\wedge \mu _{A}(b).
\end{eqnarray*}

Similarly we can show that $\gamma _{A}(ab)\leq \gamma _{A}(a)\vee \gamma
_{A}(b)$ and therefore $A=(\mu _{A},\gamma _{A})$ is an intuitionistic fuzzy
bi-ideal of $S$.
\end{proof}

\begin{theorem}
An AG-groupoid $S$ is an intra-regular if and only if for each
intuitionistic fuzzy right $($left, two-sided$)$ ideal $A=(\mu _{A},\gamma
_{A})$ of $S,$ $A(a)=A(a^{2})$ for all $a$ in $S.$
\end{theorem}

\begin{proof}
Assume that $S$ be an intra-regular AG-groupoid and let $A=(\mu _{A},\gamma
_{A})$ be an intuitionistic fuzzy right ideal of $S.$ Let $a\in S,$ then
there exists $x,y\in S$ such that $a=(xa^{2})y$.%
\begin{eqnarray*}
\mu _{A}(a) &=&\mu _{A}((xa^{2})y)=\mu _{A}((xa^{2})(ey))=\mu
_{A}((ye)(a^{2}x))=\mu _{A}(a^{2}((ye)x)) \\
&\geq &\mu _{A}(a^{2})=\mu _{A}(aa)\geq \mu _{A}(a).
\end{eqnarray*}

Similarly we can show that $\gamma _{A}(a)=\gamma _{A}(a)$ and therefore $%
A(a)=A(a^{2})$ holds for all $a$ in $S.$

Conversely, assume that for any intuitionistic fuzzy right ideal $A=(\mu
_{A},\gamma _{A})$ of $S,$ $A(a)=A(a^{2})$ holds for all $a$ in $S.$ As $%
a^{2}S$ is a right and also a left ideal of $S,$ then by Lemma \ref{00}, $%
\chi _{a^{2}S}=(\mu _{\chi _{a^{2}S}},\gamma _{\chi _{a^{2}S}})$ is an
intuitionistic fuzzy right and an intuitionistic fuzzy left ideal of $S$ and
therefore by given assumption and using the fact that $a^{2}\in a^{2}S$, we
have $\mu _{\chi _{a^{2}S}}(a)=\mu _{\chi _{a^{2}S}}(a^{2})=1$ and $\gamma
_{\chi _{a^{2}S}}(a)=\gamma _{\chi _{a^{2}S}}(a^{2})=0,$ which implies that $%
a\in a^{2}S.$ Now by using $(4)$ and $(2),$ we have $a\in (Sa^{2})S$ and
therefore $S$ is an intra-regular.
\end{proof}

\begin{theorem}
\label{left}Let $S$ be an AG-groupoid with left identity, then the following
conditions are equivalent.
\end{theorem}

$(i)$ $S$ is an intra-regular.

$(ii)$ Every intuitionistic fuzzy left (right, two-sided) ideal of $S$ is
idempotent.

\begin{proof}
$(i)$ $\Longrightarrow (ii):$ Let $S$ be an intra-regular AG-groupoid with
left identity and let $a\in S,$ then there exists $x,y\in S$ such that $%
a=(xa^{2})y$. Now by using $(4),$ $(1)$ and $(3),$ we have%
\begin{equation*}
a=(x(aa))y=(a(xa))y=(y(xa))a=((ex)(ya))a=((ay)(xe))a=(((xe)y)a)a.
\end{equation*}%
Let $A=(\mu _{A},\gamma _{A})$ be an intuitionistic fuzzy left ideal of $S,$
then by using Lemma \ref{00}, we have $\mu _{A}\circ \mu _{A}\subseteq \mu
_{A}$ and also we have%
\begin{equation*}
(\mu _{A}\circ \mu _{A})(a)=\dbigvee\limits_{a=(((xe)y)a)a}\{\mu
_{A}(((xe)y)a)\wedge \mu _{A}(a)\}\geq \mu _{A}(a)\wedge \mu _{A}(a)=\mu
_{A}(a).
\end{equation*}

Which implies that $\mu _{A}\circ \mu _{A}\supseteq \mu _{A}$ and similarly
we can get $\gamma _{A}\circ \gamma _{A}\subseteq \gamma _{A}$. Now by using
Lemma \ref{00}, $\mu _{A}\circ \mu _{A}\subseteq \mu _{A}$ and $\gamma
_{A}\circ \gamma _{A}\supseteq \gamma _{A}.$ Thus that $A=(\mu _{A},\gamma
_{A})$ is idempotent and by using Lemma \ref{iff}, every intuitionistic
fuzzy right and two-sided is idempotent.

$(ii)$ $\Longrightarrow (i)$

Assume that every left ideal of an AG-groupoid $S$ with left identity is
idempotent and let $a\in S$. Since $Sa$ is a left ideal of $S$, therefore by
Lemma \ref{00}, its characteristic function $\chi _{Sa}=(\mu _{\chi
_{Sa}},\gamma _{\chi _{Sa}})$ is an intuitionistic fuzzy left ideal of $S.$
Since $a\in $ $Sa$ therefore $\mu _{\chi _{Sa}}(a)=1.$ and $\gamma _{\chi
_{Sa}}(a)=0$. Now by using the given assumption and Lemma \ref{000}, we have 
\begin{equation*}
\mu _{\chi _{Sa}}\circ \mu _{\chi _{Sa}}=\mu _{\chi _{Sa}}\text{ and }\mu
_{\chi _{Sa}}\circ \mu _{\chi _{Sa}}=\mu _{\chi _{(Sa)^{2}}}.
\end{equation*}

Thus we have $(\mu _{\chi _{(Sa)^{2}}})(a)=(\mu _{\chi _{Sa}})(a)=1$ and
similarly we can get $(\gamma _{\chi _{(Sa)^{2}}})(a)=(\gamma _{\chi
_{Sa}})(a)=0,$ which implies that $a\in (Sa)^{2}.$ Now by using $(1),$ $(2)$
and $(3),$ we have%
\begin{eqnarray*}
a &\in &(Sa)^{2}=(Sa)(Sa)=((Sa)a)S\subseteq
((Sa)((Sa)(Sa)))S=((Sa)((SS)(aa)))S \\
&=&((Sa)(Sa^{2}))S=((a^{2}S)(aS))S=(((aS)S)a^{2})S\subseteq (Sa^{2})S.
\end{eqnarray*}

Which shows that $S$ is an intra-regular.

Let every right ideal of an AG-groupoid $S$ with left identity is idempotent 
$a\in S$. Clearly $a^{2}S$ is a right ideal of $S$, therefore by Lemma \ref%
{00}, $\chi _{a^{2}S}=(\mu _{\chi _{a^{2}S}},\gamma _{\chi _{a^{2}S}})$ is
an intuitionistic fuzzy right ideal of $S.$ As $a\in a^{2}S$ therefore $\mu
_{\chi _{a^{2}S}}(a)=1$ and $\gamma _{\chi _{a^{2}S}}(a)=0$. Now by using
the given assumption and Lemma \ref{000}, we have 
\begin{equation*}
\mu _{\chi _{a^{2}S}}\circ \mu _{\chi _{a^{2}S}}=\mu _{\chi _{a^{2}S}}\text{
and }\mu _{\chi _{a^{2}S}}\circ \mu _{\chi _{a^{2}S}}=\mu _{\chi _{a^{2}S}}.
\end{equation*}

Thus we get $(\mu _{\chi _{a^{2}S}})(a)=(\mu _{\chi _{a^{2}S}})(a)=1$ and
similarly we can get $(\gamma _{\chi _{a^{2}S}})(a)=(\gamma _{\chi
_{a^{2}S}})(a)=0,$ which implies that $a\in (a^{2}S)^{2}.$ Now by using $(3)$
and $(1),$ we have%
\begin{equation*}
a\in
(a^{2}S)^{2}=(a^{2}S)(a^{2}S)=(Sa^{2})(Sa^{2})=((Sa^{2})a^{2})S\subseteq
(Sa^{2})S.
\end{equation*}

Which shows that $S$ is an intra-regular.
\end{proof}

Note that if an AG-groupoid $S$ contains a left identity, then $S=S\circ S.$

\begin{theorem}
For an AG-groupoid $S$ with left identity$,$ the following conditions are
equivalent.
\end{theorem}

$(i)$ $S$ is an intra-regular.

$(ii)$ $A=(S\circ A)^{2},$ where $A=(\mu _{A},\gamma _{A})$ is any
intuitionistic fuzzy left (right, two-sided) ideal of $\ S$.

\begin{proof}
$(i)$ $\Longrightarrow (ii)$

Let $S$ be a an intra-regular AG-groupoid and let $A=(\mu _{A},\gamma _{A})$
be any intuitionistic fuzzy left ideal of $S,$ then it is easy to see that $%
S\circ A$ is also an intuitionistic fuzzy left ideal of $S.$ Now by using
Theorem \ref{left}, $S\circ A$ is idempotent and therefore, we have%
\begin{equation*}
(S\circ A)^{2}=S\circ A\subseteq A.
\end{equation*}

Now let $a\in S,$ since $S$ is an intra-regular therefore there exists $x\in
S$ such that $a=(xa^{2})y$ and by using $(4),$ $(3)$ and $(1),$ we have%
\begin{eqnarray*}
a &=&(x(aa))y=(a(xa))y=(((xa^{2})y)(xa))(ey)=(ye)((xa)((xa^{2})y)) \\
&=&(xa)((ye)((xa^{2})y))=(xa)(((ye)(x(aa)))y)=(xa)(((ye)(a(xa)))y) \\
&=&(xa)((a((ye)(xa)))y)=(xa)((y((ye)(xa)))a)=(xa)p
\end{eqnarray*}

where $p=((y((ye)(xa)))a)\ $and therefore, we have%
\begin{eqnarray*}
(S\circ \mu _{A})^{2}(a)
&=&\dbigvee\limits_{a=(xa)((y((ye)(xa)))a)}\{(S\circ \mu _{A})(xa)\wedge
(S\circ \mu _{A})((y((ye)(xa)))a)\} \\
&\geq &(S\circ \mu _{A})(xa)\wedge (S\circ \mu _{A})((y((ye)(xa)))a) \\
&=&\dbigvee\limits_{xa=xa}\{S(x)\wedge \mu _{A}(a)\}\wedge
\dbigvee\limits_{p=(y(y(xa)))a}\{S(y((ye)(xa)))\wedge \mu _{A}(a)\} \\
&\geq &S(x)\wedge \mu _{A}(a)\wedge S(y((ye)(xa)))\wedge \mu _{A}(a)=\mu
_{A}(a).
\end{eqnarray*}%
Similarly we can get $(S\circ \gamma _{A})^{2}(a)\leq \gamma _{A}(a)$, which
implies that $(S\circ A)^{2}\supseteq A.$ Thus we get the required $%
A=(S\circ A)^{2}$.

$(ii)$ $\Longrightarrow (i)$

Let $A=(S\circ A)^{2}$ holds for any intuitionistic fuzzy left ideal $A=(\mu
_{A},\gamma _{A})$ of$\ S,$ then by using Lemma \ref{as} and given
assumption, we have%
\begin{equation*}
\mu _{A}=(S\circ \mu _{A})^{2}\subseteq \mu _{A}^{2}=\mu _{A}\circ \mu
_{A}\subseteq S\circ \mu _{A}\subseteq \mu _{A}.
\end{equation*}

Which shows that $\mu _{A}=\mu _{A}\circ \mu _{A}$, similarly $\gamma
_{A}=\gamma _{A}\circ \gamma _{A}$ and therefore $A=A\circ A.$ Thus by using
Lemma \ref{left}, $S$ is an intra-regular.

Let $A=(S\circ A)^{2}$ holds for any intuitionistic fuzzy right ideal $%
A=(\mu _{A},\gamma _{A})$ of $\ S,$ then by using Lemma \ref{as}, given
assumption and $(1)$, we have%
\begin{equation*}
\mu _{A}=(S\circ \mu _{A})^{2}=((S\circ S)\circ \mu _{A})^{2}=((\mu
_{A}\circ S)\circ S)^{2}\subseteq \mu _{A}^{2}=\mu _{A}\circ \mu
_{A}\subseteq \mu _{A}\circ S\subseteq \mu _{A}.
\end{equation*}

Which shows that $\mu _{A}=\mu _{A}\circ \mu _{A}$, similarly $\gamma
_{A}=\gamma _{A}\circ \gamma _{A}$ and therefore $A=A\circ A.$ Thus by using
Lemma \ref{left}, $S$ is an intra-regular.
\end{proof}

\end{document}